\documentclass[12pt,reqno,a4paper]{amsart}
\usepackage{verbatim,rcs,cite}
\usepackage{amssymb,amsfonts,latexsym,mathrsfs}
\usepackage{amsmath}
\usepackage{amsthm}
\usepackage{mathrsfs}
\usepackage[all]{xy}
\usepackage{pstricks}
\newtheorem{theorem}{Theorem}[section]
\newtheorem{lemma}[theorem]{Lemma}
\newtheorem{proposition}[theorem]{Proposition}
\newtheorem{corollary}[theorem]{Corollary}

\theoremstyle{remark}
\newtheorem{example}[theorem]{Example}
\newtheorem{remark}[theorem]{Remark}
\theoremstyle{definition}
\newtheorem{definition}[theorem]{Definition}

\DeclareMathOperator{\core}{core}

\DeclareMathOperator{\Sym}{Sym}  
\DeclareMathOperator{\Alt}{Alt}
\DeclareMathOperator{\Syl}{Syl}

\DeclareMathOperator{\T}{\mathcal{T}}
\DeclareMathOperator{\I}{\mathcal{I}}

\RCS $Revision:  $ \RCSdate $Date: 2008/01/29 13:25:27 $

\title[Minimal Degrees for Irreducible Coxeter Groups]{Minimal Faithful Permutation Degrees for Irreducible Coxeter Groups and Binary Polyhedral Groups}

\date{January 2013}

\author{Neil Saunders}

\dedicatory{\upshape
School of Mathematics \\
University of Bristol, BS8 1TW\\ United Kingdom\\[.5em]
{\it E-mail address:} \ \texttt{neil.saunders@bristol.ac.uk} \\[1.5em]
Dedicated to the memory of Teddy Saunders}

\begin{document}


\thanks{\noindent{The author thanks the Heilbronn Institute for Mathematical Research for generously supporting this research}}

\maketitle

\begin{abstract}
The minimal faithful degree of a finite group $G$, denoted by $\mu(G)$, is the least non-negative integer $n$ such that $G$ embeds inside
$\Sym(n)$. In this article we calculate the minimal faithful permutation degree for all of the irreducible Coxeter groups. We also exhibit new examples of finite groups that possess a quotient whose minimal degree is strictly greater than that of the group.
\end{abstract}

\section{Introduction}

The minimal faithful permutation degree of a finite group $G$ is the least non-negative integer $n$ such that $G$ embeds in the
symmetric group $\Sym(n)$; this invariant is denoted $\mu(G)$. Since every permutation representation is afforded by a finite collection of subgroups $\{G_1,\ldots,G_l\}$, where $l \geq 1$ and the $G_i$ are the point stabilisers of the orbits of $G$, we may express $\mu(G)$ as the smallest value of $\sum_{i=1}^{l}|G:G_i|$ such that the intersection of the cores of the $G_i$ is trivial; recall $\core(G_i)=\bigcap_{g \in G}G_{i}^{g}$ (see \cite{J71}, \cite{S08}). It is common to call a collection of subgroups which furnishes $\mu(G)$ {\it the minimal representation} (though it may not in general be unique). We remark from the outset that the minimal faithful permutation degree need not be given by a transitive representation; however if the collection comprises just one subgroup $G_1$, then the representation is transitive and so by faithfulness, $G_1$ is a {\it core-free} subgroup of $G$. \vspace{6pt}

The main result of this paper is the determination of $\mu(W)$ for $W$ an irreducible Coxeter Group. Table \ref{table:CoxeterDegree} below exhibits these degrees and for comparison, we include the size of the respective root systems. It is well-known that an irreducible Coxeter group acts faithfully on its root system and faithfully on any orbit therein, thus $\mu(W)$ is always bounded above by the size of an orbit in its root system. It is the case that for the exceptional Coxeter groups, the minimal degree equals the size of an orbit in the root system precisely when the minimal faithful representation is transitive as will become apparent later. \vspace{6pt}

\begin{table}[h] \label{table:CoxeterDegree}
\caption{\textbf{Minimal Degrees for Irreducible Coxeter Groups}} \centering
\begin{tabular} {|c| c | c | }
\hline
$ \bf{W}$ & $\bf{\mu(W)}$  & $ \bf{ |\Phi|}$ \\\hline

  $A_n$ & $n+1$ & $n(n+1)$ \\
  \hline
  $B_n$ & $2n$  & $2n^2$\\
  \hline
  $D_n$ & $ 2n (n\geq 4)$ & $2n(n-1)$\\
  \hline
  $I_m$ & $\mu(C_m)$ & $12$ \\
  \hline
  $E_6$ & $27$ & $72$\\
  \hline
  $E_7$ & $30$ &$126$ \\
  \hline
  $E_8$ & $240$ & $240$ \\
  \hline
  $F_4$ & $24$ & $48$ \\
  \hline
  $H_3$ & $7$ & $30$\\
  \hline
  $H_4$ & $120$ &$120$ \\
  \hline

\end{tabular}
\end{table}

The proofs of the minimal degrees for types $F_4$ and $H_4$ required calculating central products of binary polyhedral groups. These calculations were surprising as they provided new examples of so-called {\it exceptional permutation groups}: these are groups $G$ that possess a quotient $G/N$ such that $\mu(G/N) > \mu(G)$; the quotient $G/N$ and $N$ are both called {\it distinguished}. They are somewhat pathological in the theory of permutation groups as it is in some sense 'harder' to faithfully represent a quotient of a group than is it to represent the group itself. This phenomenon was first elucidated by Neumann in \cite{N86} and again by Easdown and Praeger in \cite{EP88}; the examples provided in those articles comprised $p$-groups only. In Section \ref{section:exceptional}, we provide a general family of exceptional non-$p$-groups inspired by the $F_4$ and $H_4$ calculations and the results of the aforementioned articles.   \vspace{6pt}

The motivation for considering the Coxeter groups was the following inequality: for two finite groups $G$ and $H$ we always have \begin{equation} \mu(G \times H) \leq \mu(G) + \mu(H).
\label{eq:directsum} \end{equation} Johnson \cite{J71} and Wright \cite{W75} first determined conditions when \eqref{eq:directsum} is an equality, however they were unaware of examples when equality was strict. Indeed, in the closing remarks of \cite{W75}, due to the absence of any examples, Wright asked whether (\ref{eq:directsum}) is an equality for all finite groups. An example where \eqref{eq:directsum} is a strict inequality was subsequently provided and was attached as an addendum to that paper. \vspace{6pt}

In \cite{S10}, it was proved than $10$ is the smallest degree of a direct product where \eqref{eq:directsum} is a strict inequality. Here, the group $G$ could be taken to be $G(2,2,5)$ and $H$ to be the centraliser of $G$ in $\Sym(10)$ which is cyclic of order $2$. Moreover in \cite{S07,S08}, the author studied the minimal degrees of the complex reflection groups $G(p,p,q)$, where $p$ and $q$ are primes and constructed an infinite class of examples where strict inequality holds (see Section \ref{section:exceptional} for a full definition of this family of groups). It was always the case that for $G=G(p,p,q)$, where $p$ and $q$ are distinct primes satisfying certain other conditions, we have $$\mu(G)=\mu(G \times C_{\Sym(\mu(G))}(G))$$ where $C_{\Sym(\mu(G))}(G)$, the centraliser of the embedded image, was isomorphic to a cyclic group of order $p$. \vspace{6pt}

Since it is well-known that $G(2,2,n)$ is isomorphic to the Coxeter group $W(D_n)$, the pursuit of more examples of \eqref{eq:directsum} being a strict inequality lead the author to consider the irreducible Coxeter groups whose minimal degrees were not readily available in the literature.  \vspace{6pt}

\subsection{Organisation of the paper} This article is a self-contained account of the minimal degrees for the irreducible Coxeter groups. In Section \ref{section:background}, we provide some relevant theorems and examples which will be called upon repeatedly. Section \ref{section:classicalgroups} deals with the classical Coxeter groups of types $A$, $B$ and $D$. Section \ref{section:realreflections} deals with the exceptional groups of types $F$ and $H$: here we need to calculate the minimal degrees of the binary polyhedral groups and the central products of these groups with themselves, which turn out to provide new examples of exceptional groups. In Section \ref{section:exceptionalcoxeter}, we calculate the minimal degrees of the groups of type $E$ and in Section \ref{section:directproduct} we make some summary remarks about minimal degrees of arbitrary direct products of irreducible Coxeter groups. In Section \ref{section:exceptional}, we discover new families of exceptional permutation groups based on the calculations of Section \ref{section:realreflections}.

\section{Background Results and Examples} \label{section:background} 

We give a series of theorems and examples that we will implicitly, but frequently, use throughout the sequel. First, we give a theorem due to Karpilovsky \cite{K70}, which also serves as an introductory example; the proof can be found in \cite{J71} or \cite{K70}.
\begin{theorem} \label{theorem:abelian}
Let A be a finite abelian group and let $A \cong A_1 \times \ldots \times A_n$ be its direct product decomposition into non-trivial cyclic
groups of prime power order. Then $$\mu(A)= a_1 + \ldots + a_n,$$ where $|A_i|=a_i$ for each $i$.
\end{theorem}

In \cite{W75} Wright proved the following: 

\begin{theorem} \label{theorem:nilpotent}
Let $G$ and $H$ be non-trivial nilpotent groups. Then $\mu(G \times H)=\mu(G) + \mu(H)$.
\end{theorem}

Further in \cite{W75}, Wright constructed a class of finite groups $\mathscr{C}$ with the defining property that any $G \in \mathscr{C}$, there exists a
nilpotent subgroup $G_1$ of $G$ such that $\mu(G_1)=\mu(G)$. It is a consequence of Theorem \ref{theorem:nilpotent} that $\mathscr{C}$ is
closed under direct products and so \eqref{eq:directsum} is an equality for any two groups in $\mathscr{C}$. Wright proved that
$\mathscr{C}$ contains all nilpotent, symmetric, alternating and dihedral groups, and  in \cite{S10}, it was proved that if $\mu(G) \leq 6$, then $G \in \mathscr{C}$; however the extent of it is still an open problem. In \cite{EP88}, Easdown and Praeger showed that \eqref{eq:directsum} is an equality whenever $G$ and $H$ are finite simple groups. \vspace{6pt}

\vspace{6pt}

The next two examples involve the generalised quaternion $2$-group. The calculations here will be repeatedly referred to throughout later sections. 

\begin{example} \label{example:quaternions}
Let $Q_{2^n}$ be the generalised quaternion $2$-group of order $2^n$, which may be given the following presentation: $$Q_{2^n}= \langle x,y \, | \, x^{2^{n-1}}=1, \, x^{2^{n-2}}=y^2, \, x^y=x^{-1} \rangle.$$ It has the property that every subgroup is normal and centre, $\langle y^2 \rangle$, is the unique minimal normal subgroup and so any two non-trivial subgroups intersect at the centre. Thus the minimal faithful degree is given by the Cayley representation and so $\mu(Q_{2^n})=2^n.$  Since the central element is the unique involution of $Q_{2^n}$, it is common to write $y^2=-1$ and so $\langle y^2 \rangle=\{\pm 1\}$; we will do this in the next example and in later sections.
\end{example}

\begin{example} \label{example:quatcentral}
Let $m$ and $n$ be positive integers greater than $2$ with $m \geq n$. For the quaternion groups $Q_{2^n}$ and $Q_{2^m}$, consider the subgroup of the direct product $Q_{2^n} \times Q_{2^m}$ generated by the diagonal element of the centre $\langle (-1,-1) \rangle$. The {\it central product} $Q_{2^n} \circ Q_{2^m}$  of $Q_{2^n}$ with $Q_{2^m}$ is the quotient $$Q_{2^n} \circ Q_{2^m}:= Q_{2^n} \times Q_{2^m}/ \langle (-1,-1) \rangle.$$ (In general, the central product of a group by itself is formed by taking the quotient of the direct product of the group with itself by the diagonal subgroup of the centre). \vspace{6pt} 

Since $Q_{2^n}$ naturally embeds in $Q_{2^m}$, we have an action of $Q_{2^n} \times Q_{2^m}$ on $ Q_{2^m}$ defined as follows: 
\begin{eqnarray*}
Q_{2^n} \times Q_{2^m}  &\longrightarrow& \Sym( Q_{2^m}); \\ 
(g,h) &\mapsto& \sigma_{(g,h)}: x \mapsto g^{-1}xh, 
\end{eqnarray*}
for all $x$ in $Q_{2^m}$. The kernel of this action is the diagonal subgroup of the centre $ \langle (-1,-1) \rangle$, and so there is an embedding of $Q_{2^n} \circ Q_{2^m}$ in $\Sym(Q_{2^m})\cong \Sym(2^m)$; therefore $\mu(Q_{2^n} \circ Q_{2^m}) \leq 2^m$. On the other hand, it can readily be seen that $Q_{2^m}$ embeds in $Q_{2^n} \circ Q_{2^m}$ and so, by Example \ref{example:quaternions}, we have $\mu(Q_{2^n} \circ Q_{2^m})=2^m$.

\end{example}

\begin{remark}
In \cite{EP88}, the authors claim that $Q_{2^n} \circ Q_{2^m}$ is a distinguished quotient of  $Q_{2^n} \times Q_{2^m}$, however their calculations, which is what Example \ref{example:quatcentral} is based on, show that this is not the case. 
\end{remark}

\section{ The Classical Groups $W(A_n), W(B_n)$ and $W(D_n)$} \label{section:classicalgroups}

The Coxeter group $W(A_n)$ is the symmetric group $\Sym(n+1)$ and so it has minimal degree $n+1$. The Coxeter group $W(B_n)$ is the full wreath
product $C_2 \wr \Sym(n)$ which acts faithfully as signed permutations on the set $\{\pm1, \pm2, \ldots, \pm n\}$, which shows that $\mu(W(B_n)) \leq
2n$. On the other hand, the base group is an elementary abelian $2$-group of rank $n$ which has minimal degree $2n$ by Theorem \ref{theorem:abelian}. Therefore  $\mu(W(B_n))=2n$. \vspace{6pt}

The calculation for the Coxeter group $W(D_n)$ is a little harder and so we appeal to an argument given in \cite{KL90} to do most of this
calculation. We first need to establish some definitions and preliminary results regarding permutation actions.

\begin{definition} \label{proposition:permutationmodule}
Let $p$ be a prime and $n$ an integer. A {\it permutation module} (depending on $n$ and $p$) for the symmetric group $\Sym(n)$ is the direct sum of $n$ copies of the
cyclic group of order $p$ denoted by $\mathbb{F}_{p}^{n}$ where $\Sym(n)$ acts via permuting coordinates. Define two submodules of
$\mathbb{F}_{p}^{n}$,
\begin{eqnarray*}
U &=& \{(a_1,a_2, \ldots,a_n) \in \mathbb{F}_{p}^{n} \ | \ \sum_{i=1}^{n}a_i=0 \}, \\
V &=& \{(a,a,\ldots,a) \ | \ a \in \mathbb{F}_p \}.
\end{eqnarray*}
\end{definition}

In the above definition, $U$ is a submodule of dimension $(n-1)$ called the {\it deleted permutation module}.
 The next result is well-known and easily verified by a direct calculation.
\begin{proposition} \label{proposition:permmod}
Let $n \geq 3$ an integer. The only proper submodules of $\mathbb{F}_{p}^{n}$ invariant under the action of the alternating group $\Alt(n)$ are $U$ and $V$.
\end{proposition}

The Coxeter group $W(D_n)$ is the split extension of the deleted permutation module $U$, when $p=2$, with the symmetric group $\Sym(n)$. It can
be also realized as the group of even signed permutations of the set $\{\pm 1, \ldots, \pm n \}$ and so is a subgroup of index $2$ in the group
$W(B_n)$.

\begin{theorem} \label{theorem:Dn}
Let $n \geq 4$ be an integer. We have 
$\mu(W(D_n))= 2n.$ 
\end{theorem}

\begin{proof} 

For $n=4$, we use the realisation of $W(D_4)$ as the group of even signed permutations on the set $\{\pm 1, \pm 2, \pm 3, \pm4 \}$ (using
commas in the cycle notation this time), so
$$W(D_4)=\langle (1, \ -2)(2, \ -1), (1, \ 2)(-1, \ -2), (1, \ 3)(-1, \ -3), (1, \ 4)(-1, \ -4) \rangle.$$ Hence by taking products of conjugates of
generators, let
\begin{eqnarray*}
x &=& (1, \ -2)(2, \ -1)(3, \ -4)(4, \ -3)(1, \ -1)(3, \ -3) \\
  &=& (1, \ -2, \ -1, \ 2)(3, \ -4, \ -3, \ 4),
\end{eqnarray*} and
\begin{eqnarray*}
y &=& (1, \ -3)(3, \ -1)(1, \ -1)(2, \ -2)(2, \ -4)(4, \ -2) \\
  &=& (1, \ -3, \ -1, \ 3)(2, \ 4, \ -2, \ -4).
\end{eqnarray*}
Both lie in $W(D_4)$. But
$$x^2=y^2=(1, \ -1)(2, \ -2)(3, \ -3)(4, \ -4)$$ and $$x^y=(-3, \ -4, \ 3, \ 4)(1, \ 2, \ -1, \ -2)= x^{-1},$$ so that $Q_8$ is isomorphic to a proper subgroup of $W(D_4)$. Hence
$$8=\mu(Q_8)\leq \mu(W(D_4)) \leq \mu(W(B_4))=8,$$ so that $\mu(W(D_4))=8$. \vspace{6pt}

For $n$ greater than or equal to $5$, the proof that $\mu(W(D_n))=2n$ is a special case of \cite[Proposition 5.2.8]{KL90}, where they calculate
the minimal permutation degree of $U \rtimes \Alt(n)$ for $n\geq 5$, relying on the simplicity of $\Alt(n)$. We give an explicit proof here for
convenience. \vspace{6pt}

Define a proper subgroup of $W(D_n)$, $H:= U \rtimes \Alt(n)$. We will show that $\mu(H) \geq 2n$ for $n \geq 5$. Let $H$ act faithfully (but not necessarily transitively) on a set $X$: we will show that $X$ has size at least $2n$. Since $U$ is normal in $H$, the $U$-orbits form an $H$-invariant partition of $X$. Consider the induced action of $H$ on this partition and let $K$ be the kernel of this action; by the simplicity of $\Alt(n)$, either $K=\{1\}$ or $K=\Alt(n)$. \vspace{6pt}

Suppose $K=\{1\}$. Let $Y$ be any non-trivial $U$-orbit and let $\{Y_1,\ldots Y_r\}$ be the orbit of $Y$ under the action of $\Alt(n)$. It follows that $r \geq n$ since $\mu(\Alt(n))=n$ and hence $|X| \geq 2r \geq 2n,$ since $Y$ has size at least $2$ being an non-trivial $U$-orbit, and we are done. \vspace{6pt}

Henceforth we may suppose that $K=\Alt(n)$. Again, let $Y$ be any non-trivial $U$-orbit; here we have $Y^{\Alt(n)}=Y$. Choose $y \in Y$ and let $U_y$ be its point stabiliser in $U$. Since $U$ is abelian, $U_y$ is the kernel (or the set-wise stabiliser) of the $U$-action on $Y$, and moreover, it follows that $U_y$ is normal in $H$. Also observe that since $|Y| \geq 2$, $U_y$ is
a properly contained in $U$ and is stable under the conjugation action of $\Alt(n)$. By Proposition \ref{proposition:permmod}, $U_y$ is trivial or equals $V$, so
that $|U_y|=1$ or $2$. If $|U_y|=1$, then $$|X| \geq |Y|= |U|=2^{n-1} \geq 2n,$$ and we are done. If $|U_y|=2$, then by faithfulness, there must be at least one other non-trivial $U$-orbit (since the kernel of the $U$-action on $Y$ is $U_y$), so $$|X| \geq |Y| + 2=|U:U_y| + 2= 2^{n-2}+2 \geq 2n,$$ and again we are done. This completes the proof.
\end{proof}

\begin{remark} 
When $n=2$ or $3$, $W(D_n)$ is not considered irreducible or is isomorphic to a Coxeter group of type $A$, so we omitted them above. However these groups can be abstractly defined and when $n=2$, this group is the Klein four group and so $\mu(W(D_2))=4$, and for $n=3$, we have $W(D_3) \cong (C_2 \times C_2) \rtimes \Sym(3)$ which
is well-understood to be isomorphic to the symmetric group $\Sym(4)$ via the representation $$\langle (1 \ 2)(3 \ 4), (1\ 4)(2 \ 3), (1 \ 2 \ 3), (1 \ 3) \rangle,$$ so $\mu(W(D_3))=4$. \vspace{6pt}

\end{remark}

\section{The Real Reflection Subgroups: $W(H_3), W(H_4)$ and $W(F_4)$} \label{section:realreflections}

The Coxeter groups $W(H_3), W(H_4)$ and $W(F_4)$ can be realised as reflection groups of real $4$-dimensional space, that is, finite subgroups of
$O_{4}(\mathbb{R})$. The reflection subgroups of real three and four dimensional space have been studied extensively and we refer to \cite{TL07} for a comprehensive treatment.
\vspace{6pt}

We first deal with the Coxeter group $W(H_3)$, making use of the following result.

\begin{theorem}\cite[Theorem 3.1]{EP88} \label{theorem:simplegroups}
Let $S_1 \times \ldots \times S_r$ be a direct product of simple groups. Then $$\mu(S_1 \times \ldots \times S_r)=\mu(S_1) + \ldots +
\mu(S_r).$$
\end{theorem}

The Coxeter group $W(H_3)$ is isomorphic to the direct product $C_2 \times \Alt(5)$ (see \cite{TL07}) which are clearly simple groups. Moreover,
it is easy to see that the minimal degree of $\Alt(5)$ is $5$ and so by Theorem \ref{theorem:simplegroups}, we have:

\begin{proposition} 
The minimal degree of  $\mu(W(H_3))$ is 7.

\end{proposition}

For the groups $W(H_4)$ and $W(F_4)$, we need some results relating to the finite subgroups of the quaternions. Specifically we will need to calculate the minimal degrees of the central products of the binary tetrahedral group with itself and the binary icosahedral group with itself. This will require a rather lengthy digression into the minimal degrees of the binary polyhedral groups, however new results about distinguished quotients arise which will be further examined in Section \ref{section:exceptional}.

\subsection{The Quaternions and the Binary Polyhedral Groups} \label{subsection:polyhedral}
The quaternions are defined as 
 $$\mathbb{H}=\big\{a+bi+cj+dk \ | \ a,b,c,d \in \mathbb{R}, \ i^2=j^2=k^2=ijk=-1 \big\}.$$ This may be regarded as a  $4$-dimensional real vector space as well as an algebra equipped with a norm function. For every quaternion $h=a+bi+cj+dk$, the norm of $h$ is $a^2+b^2+c^2+d^2$. The
set of quaternions of norm $1$ form a subgroup of the multiplicative group $\mathbb{H}^*$ called the {\it unit quaternions}, denoted by $S^3$. It is well-known that all finite subgroups of $\mathbb{H}^*$ are contained in $S^3$. There is a surjection of the semidirect product $(S^3 \times S^3) \rtimes C_2$ onto the orthogonal group $O_{4}(\mathbb{R})$ (where $C_2$ acts
by interchanging components). The kernel of this homomorphism is the diagonal subgroup of the centre of $S^3 \times S^3$ and so we have an
isomorphism $(S^3 \circ S^3) \rtimes C_2 \cong O_{4}(\mathbb{R})$ where $\circ$ denotes central product. Thus the finite reflection subgroups of
$O_{4}(\mathbb{R})$ are tightly controlled by the finite subgroups of $S^3$, of which, there only five classes. Here is a classification of the finite subgroups of $S^3$: since we will not deal with the first two on the list, we refer the reader to \cite[Theorem 5.14]{TL07} for their definition and for a proof of the following proposition. 

\begin{proposition}
Every finite subgroup of $S^3$  is conjugate in $S^3$ to one of the following groups:
\begin{enumerate}
\item the cyclic group $\mathcal{C}_m$ of order $m$, 
\item the binary dihedral group $\mathcal{D}_m$ of order $4m$, 
\item the binary tetrahedral group $\mathcal{T}$ of order $24$, 
\item the binary octahedral group $\mathcal{O}$ of order $48$,
\item the binary icosahedral group $\mathcal{I}$ of order $120$.
\end{enumerate}
\end{proposition}

We also make the following easy, but useful observation: 

\begin{lemma} \label{lemma:unique2}
Let $G$ be any subgroup of $S^3$ of even order. Then $\langle -1 \rangle \cong C_2$ is the unique subgroup of order $2$ in $G$.
\end{lemma}

We now define and further explore the structure of the groups $\T$ and $\I$ further, with a view to calculate the minimal degrees of the central products of these groups with themselves. For what follows, let $Q_8$ denote the multiplicative group $\{ \pm 1, \pm i, \pm j, \pm k \}$, the quaternion group of order $8$ and let 
$\omega=\tfrac12(-1+i+j+k)$. Some easy calculations verify that $\omega$ has order $3$ and normalises $Q_8$; indeed one sees that $\omega$ permutes $i,j$ and $k$ in a $3$-cycle. The binary tetrahedral
group is defined to be $\mathcal{T}:= Q_8\langle \omega \rangle \cong Q_8 \rtimes C_3.$ We observe that $\mathcal{T}$ has order $24$ and can be generated by $i$ and
$\omega$. Moreover some elementary Sylow theory verifies that the normaliser of $\langle \omega \rangle$ in $\T$ is $\langle \omega \rangle \times \{\pm1\}$; this fact will be needed later.  \vspace{6pt}

Let $\tau:=\frac{1}{2}(1+\sqrt{5})$, the golden ratio, and let $\sigma:=\frac{1}{2}(\tau^{-1}+i+\tau j)$. A direct calculation shows that $\sigma$ has order $5$ and a few more calculations of a similar nature show that $$\tau - \tau^{-1}=1, \ \tau^{2}=\tau +1 \quad \text{and} \quad \tau^{-2}=-\tau^{-1}+1.$$ The
 binary icosahedral group is defined to be $\I := \mathcal{T}\langle \sigma \rangle$. Note that it can also be generated
by $\sigma$ and $i$ (since $\omega=i^{-1}(\sigma^3 i \sigma i)^2i$) and has order $120$. It may be given by the presentation $$\mathcal{I} \cong \langle s, t \ | \ (st)^2=s^3=t^5 \rangle,$$ by
identifying $s$ with $\frac{1}{2}(1+i+j+k) \ (=-\omega^2)$ and $t$ with $\frac{1}{2}(\tau + \tau^{-1}i+j) \, (=-\sigma^3)$ (see \cite[page
439]{S83}). For convenience in what follows, we state and prove that the alternating group $\Alt(5)$ is a quotient of the binary icosahedral group by its centre, which has implications its subgroup structure.

\begin{lemma} \label{lemma:icosahedral}
$\mathcal{I}/\{\pm 1\} \cong \Alt(5).$
\end{lemma}

\begin{proof}
A simple calculation verifies that $$\left(\tfrac12(1+i+j+k) \cdot \tfrac12(\tau + \tau^{-1}i+j)\right)^2=-1,$$ so adding the relation $(st)^2=1$ to the
presentation for $\mathcal{I}$, and using Tietze transformations, we obtain a well-known presentation for $\Alt(5)$ (see \cite{M68}):
\begin{eqnarray*}
\mathcal{I}/\{\pm 1\} &\cong& \langle s, t \ | \ (st)^2=s^3=t^5=1 \rangle \\
 &\cong& \langle s, t, a, b \ | \ (st)^2=s^3=t^5=1, a=st, b=s \rangle \\
 &\cong& \langle a,b \ | \ a^2=b^3=(ab)^5=1 \rangle \\
 &\cong& \Alt(5),
 \end{eqnarray*}
 which completes the proof.
\end{proof}

\begin{corollary} \label{corollary:oddalt}
The only subgroups of $\mathcal{I}$ of odd order are cyclic of orders $1, 3$ or $5$.
\end{corollary}

\begin{proof}
Let $H$ be a subgroup of $\mathcal{I}$ of odd order. Since $|\mathcal{I}|=2^3 \cdot 3 \cdot 5$, $|H|$ divides $15$. If $|H|=15$, then $H$ is
cyclic (well-known fact that easily follows from Sylow theory), so $$H \cong \pm H/\{\pm 1\} \hookrightarrow \Alt(5),$$ so that $$8=\mu(C_3
\times C_5)=\mu(H) \leq \mu(\Alt(5))=5,$$ a contradiction. Hence $|H|$ has order $1, 3$, or $5$.
\end{proof}

We will show that taking direct products of the binary tetrahedral, octahedral and icosahedral groups with themselves are exceptional groups, with distinguished quotients being central products. We first calculate the minimal degrees of these groups, the explicit permutation representations were calculated in Magma.

\begin{proposition}
We have
\begin{enumerate}
\item $\mu(\mathcal{T})=8$, with an explicit representation on right cosets of $\langle \omega \rangle$:
\begin{eqnarray*}
\mathcal{T} &\hookrightarrow& \Sym(8); \\
    i       &\mapsto& (1 \ 2 \ 4 \ 7)(3 \ 6 \ 8 \ 5),\\
    j       &\mapsto& (1 \ 3 \ 4 \ 8)(5 \ 7 \ 6 \ 2),\\
    \omega       &\mapsto& (2 \ 3 \ 5)(6 \ 7 \ 8).
\end{eqnarray*}

\item $\mu(\mathcal{I})=24$, with an explicit representation with an explicit action on the right cosets of $\langle t \rangle$:
\begin{eqnarray*}
\mathcal{I} &\hookrightarrow& \Sym(24); \\
     s      &\mapsto& (1 \ 2 \ 5 \ 3 \ 7 \ 4)(6 \  10 \  9 \  12 \  14 \  8)(11 \  17 \  15 \  13 \ 16 \ 18)(19 \  20 \  24 \  22 \ 21 \  23), \\
     t      &\mapsto& (1 \  3)(2 \  4 \  8 \  13 \  12 \  7 \  5 \  9 \  11 \  6)(10 \ 15 \  19 \  21 \  17 \ 14 \  18 \ 22 \  20 \  16)(23 \  24).
\end{eqnarray*}

\end{enumerate}
\end{proposition}

\begin{proof}
We start with the binary tetrahedral group $\mathcal{T}=Q_8 \rtimes \langle \omega \rangle$. By Example \ref{example:quaternions}, $\mu(Q_8)=8$ which is thus a lower bound for $\mu(\mathcal{T})$. The following calculation $$i^{-1}\omega i=\frac{1}{2}(1+i-j-k),$$ shows that $\langle \omega \rangle$ is not a
normal subgroup and so it must be a core-free subgroup of index $8$. Therefore $\mu(\mathcal{T})=8$. \vspace{6pt}

For the binary icosahedral group $\mathcal{I}$, we claim that $\langle \sigma \rangle$ forms a core-free subgroup of index $24$.
Indeed, since $\mathcal{I}$ is generated by $\sigma$ and $i$, we explicitly list the elements of the following groups:

\begin{eqnarray*}
\langle \sigma \rangle &=& \big\{1,\tfrac{1}{2}(\tau^{-1}+i+\tau j), \tfrac{1}{2}(-\tau+\tau^{-1}i+j), \tfrac{1}{2}(-\tau-\tau^{-1}i- j), \tfrac{1}{2}(\tau^{-1}-i-\tau j) \big \}, \\
\langle i^{-1} \sigma i \rangle &=& \big\{1,\tfrac{1}{2}(\tau^{-1}+i-\tau j), \tfrac{1}{2}(-\tau+\tau^{-1}i-j), \tfrac{1}{2}(-\tau-\tau^{-1}i + j), \tfrac{1}{2}(\tau^{-1}-i+\tau j)\big \}.
\end{eqnarray*}
 Therefore $\langle \sigma \rangle \cap \langle \sigma^i \rangle =\{1\}$ which implies $\core \langle \sigma \rangle =\{1\}$ and so
$\mu(\mathcal{I}) \leq 24$. \vspace{6pt}

On the other hand by Lemma \ref{lemma:icosahedral}, we know that $\mathcal{I}/\{\pm 1 \} \cong \Alt(5)$. Let $\varphi$ be the map from $\mathcal{I}$ to
$\Alt(5)$ with kernel $\{\pm 1\}$. Suppose there is a normal subgroup $N$ of $\mathcal{I}$ such that $-1$ is not in $N$. Then the restriction
map $\varphi': N \mapsto \varphi(N)$ has trivial kernel and so $N \cong \varphi(N)$ which is necessarily a normal subgroup of $\Alt(5)$.  By the
simplicity of $\Alt(5)$, we are forced to have $\varphi(N) = \Alt(5)$ and so $N$ has even order. However by Lemma \ref{lemma:unique2}, every
subgroup of even order contains $-1$ and so any non-trivial normal subgroup of $\mathcal{I}$ contains $-1$. This forces every minimal
representation of $\mathcal{I}$ to be transitive, that is, $\mu(\I)=|\I:H|$ for some core-free subgroup $H$. \vspace{6pt}

By the above argument, $|H|$ must be odd. Since $|\mathcal{I}|= 2^3 \cdot 3 \cdot 5$, and noting that $\mu(\mathcal{I}) \leq 24$, $H$ must
have order $5$ or $15$. Suppose for a contradiction that $|H|=15$. Then $H$ is cyclic of order $15$, so $\varphi(H) \cong H$ is a cyclic subgroup
of order $15$ inside $\Alt(5)$, which contradicts Corollary \ref{corollary:oddalt}.\vspace{6pt}

Therefore the largest core-free subgroup of $\mathcal{I}$ has order $5$, proving that $\mu(\mathcal{I})=24$.
\end{proof}

 We now calculate the minimal degrees of the direct products $\T \times \T$ and $\I \times \I$, after stating a counting
 lemma about direct products.
 
 \begin{lemma} \label{lemma:counting}
Let $H$ be a subgroup of the direct product $G_1 \times G_2$ and let $\varphi_{i}: H \longrightarrow G_i$ be the projection map for $i=1,2$. Put
$$H^1=\{ y \in G_2 \ | \ (1,y) \in H \} \quad \text{and} \quad H^2= \{x \in G_1 \ | \ (x,1) \in H \}.$$ Then $H^1$ is a subgroup of $1 \times G_2$ and
$H^2$ is a subgroup of $G_1 \times 1$ and $$|H|=|\varphi_{i}(H)||H^i|,$$ for $i=1,2$.
 \end{lemma}

\begin{proof}
This is a simple consequence of the First Isomorphism Theorem.
\end{proof}

\begin{proposition} We have
\begin{enumerate}
\item $\mu(\mathcal{T} \times \mathcal{T})=16$ \quad and  
\item $\mu(\mathcal{I} \times \mathcal{I} )= 48$.
\end{enumerate}
\end{proposition}

\begin{proof}
For the group $\mathcal{T}$, it is readily observed that it is contained in Wright's class $\mathscr{C}$ (defined in the introduction), with $Q_8$ being the nilpotent subgroup. Thus $$\mu(\mathcal{T} \times
\mathcal{T})=\mu(\mathcal{T}) + \mu(\mathcal{T})=16.$$

For the group $\mathcal{I} \times \mathcal{I}$, we first note that  $\mu(\mathcal{I} \times  \mathcal{I}) \leq 2\mu(\mathcal{I})=48$. Suppose
that $\{L_1, \ldots, L_k \}$ is a collection of subgroups of $\mathcal{I} \times \mathcal{I}$ that affords a minimal faithful representation.
Then $L_{1}^1, \ldots, L_{k}^1$ (following the notation of Lemma \ref{lemma:counting}) are subgroups of $\mathcal{I}$, so if they all have even
order, then $\langle (1,-1) \rangle$ is contained in every $L_i,$ contradicting that $\cap_{i=1}^{k}L_i$ has trivial core. \vspace{6pt}

Without loss of generality, we may assume that $L_{1}^{1}$ has odd order, so by Corollary \ref{corollary:oddalt}, $|L_{1}^{1}| \leq 5$. If
$|\varphi_{1}(L_{1})| < |\mathcal{I}|$, then $|\mathcal{I}:\varphi_{1}(L_{1})| \geq 2$, so

$$\mu(\mathcal{I} \times \mathcal{I}) \geq |\mathcal{I} \times \mathcal{I}:L_1| =  \frac{|\mathcal{I} \times \mathcal{I}|}{|\varphi_{1}(L_1)||L_{1}^{1}|} =|\mathcal{I}:\varphi_{1}(L_1)| |\mathcal{I}:L_{1}^{1}| \\
\geq 2 \cdot \frac{120}{5}=48,$$
\noindent and we are done. \vspace{6pt}

Hence we may suppose $\varphi_{1}(L_1)=\mathcal{I}$, so $|L_1|=|\mathcal{I}||L_{1}^{1}|$. In particular, $(-1,y) \in L_1$ for some $y$: if $y$ has even order then $L_1$ immediately contains a central element of order $2$, which we may without loss of generality assume is $(-1,1)$,  and if $y$ has odd
order, then $$(-1,1)=(-1,y)^{|y|} \in L_1.$$ If $k=1$, then $\langle (-1,1) \rangle \subset L_1$, contradicting that $L_1$ has trivial core; hence
$k \geq 2$. If $|L_{i}^{2}|$ is even for all $i \geq 2$, then again $\langle (1,-1) \rangle$ is contained in every $L_i,$ contradicting that $\cap_{i=1}^{k}L_i$ has trivial core. \vspace{6pt}
Without loss of generality, we may assume that $L_{2}^{2}$ has odd order, so as before, $|L_{2}^{2}| \leq 5$. Hence

$$
\mu(\mathcal{I} \times \mathcal{I}) \geq |\mathcal{I} \times \mathcal{I}:L_1| +  |\mathcal{I} \times \mathcal{I}:L_2|
 = \frac{|\mathcal{I}|^2}{|\mathcal{I}||L_{1}^{1}|} + \frac{|\mathcal{I}|}{|\varphi_{2}(L_2)|}\frac{|\mathcal{I}|}{|L_{2}^{2}|} 
\geq \frac{|\mathcal{I}|}{|L_{1}^{1}|} + \frac{|\mathcal{I}|}{|L_{2}^{2}|}
\geq 2\cdot \frac{120}{5} =48, 
$$
and the proposition is proved.
\end{proof}

To calculate the minimal degrees of the central products $\T \circ \T$ and $\I \circ \I$, we will need the following: 

\begin{lemma}\label{lemma:centralproduct}
Let $G$ be a finite group and $p$ is prime. Suppose that the centre $Z(G)$ is cyclic of order $p$ and that it is the unique minimal normal
subgroup of $G$. Then the central product $G \circ G$ has a unique minimal normal subgroup isomorphic to $C_p$, namely $Z(G) \circ Z(G)$.
\end{lemma}

\begin{proof}
Let $\overline{N}$ be a non-trivial normal subgroup of $G \circ G$. Let $N$ be the pre-image of $\overline{N}$ in $G \times G$, a normal subgroup of $G
\times G$ strictly containing the diagonal copy of $Z(G)$. \vspace{6pt}

Let $(x,g)$ be any element of $N$ not contained in the diagonal copy of $Z(G)$. If $x,g \in Z(G)$, then $(x,g)$ and this diagonal copy generate
$Z(G) \times Z(G)$, which is thus contained in $N$. Otherwise, one of $x$ and $g$, say $g$, is not contained in $Z(G)$, so there is some
element $h \in G$ such that $g^{-1}h^{-1}gh \neq 1$. But $N$ contains $=(x,g)^{-1}(x,g)^{(1,h)}=(1, g^{-1}h^{-1}gh)$ as well as its normal
closure. So $N$ contains $\{1\} \times Z(G)$ and therefore contains $Z(G) \times Z(G)$ in this case also. Passing to the quotient, we find $Z(G)
\circ Z(G)$ is contained in $\overline{N}$.
\end{proof}

\begin{lemma} \label{lemma:structure}
$\mathcal{T} \times \mathcal{T} \cong (Q_8 \times Q_8) \rtimes (C_3 \times C_3)$ and so $\mathcal{T} \circ \mathcal{T} \cong (Q_8 \circ Q_8)
\rtimes (C_3 \times C_3),$ where the semidirect products are taken under appropriate actions.
\end{lemma}
\begin{proof}
The first isomorphism is clear since under the appropriate actions,
\begin{eqnarray*}
\mathcal{T} \times \mathcal{T} &\cong& (Q_8 \rtimes C_3) \times (Q_8 \rtimes C_3) \\
&\cong& (Q_8 \times Q_8) \rtimes (C_3 \times C_3).
\end{eqnarray*}

For the second isomorphism, observe that since $Z(\mathcal{T})=Z(Q_8) \cong C_2$, we deduce that $\mathcal{T} \times \mathcal{T}$ surjects onto $(Q_8
\circ Q_8) \rtimes (C_3 \times C_3)$ with kernel $\{(z,z) \ | \ z \in Z(\mathcal{T}) \}$, the diagonal subgroup of the centre of $\mathcal{T}
\times \mathcal{T}$. Thus the second isomorphism follows.
\end{proof}

Lemma \ref{lemma:centralproduct} implies that the minimal degree of $\T \circ \T$ is obtained via a transitive representation, thus we seek the largest core-free subgroup to calculate its minimal degree.

\begin{remark}
We remark that when we know that the minimal degree is given by a transitive permutation representation, it can be easily calculated using a computer since a maximal core-free subgroup yields the minimal degree. This is the case with Propositions \ref{proposition:tetra} and  \ref{proposition:icosa}, and also in previous calculations. 
\end{remark}

\begin{proposition} \label{proposition:tetra}
We have $\mu(\T \circ \T)=24.$
\end{proposition}

\begin{proof}
Define a map $\phi: \T \times \T \longrightarrow \Sym(\T)$ by sending the pair $(g,h)$ to the permutation $x \mapsto g^{-1}xh$ for all $x$ in $\T$. A quick calculation shows that the kernel of this map is the diagonal subgroup of $Z(\T) \times Z(\T)$ and so $\T \circ \T$ embeds in $\Sym(\T)=\Sym(24)$; thus $\mu(\T \circ \T) \leq 24$. Now, since we can easily find a copy of $\T \times C_3$ in $\T \circ \T$,  we have $11 =\mu(C_3 \times \T) \leq \mu(\T \circ \T)$. Since the minimal degree of $\T \circ \T$ must be afforded by a core-free subgroup, we make the following claim: \\
\\
\noindent \underline{Claim:} If $L$ is a core-free subgroup of $\mathcal{T}  \circ  \mathcal{T}$, then $|\mathcal{T}  \circ  \mathcal{T}:L| \geq 24$. \\
\\
Suppose for a contradiction that $\core(L)=\{1\}$ and $|\mathcal{T}  \circ  \mathcal{T}:L| < 24$. Then, since $|\mathcal{T}  \circ
\mathcal{T}|=2^5 \cdot 3^2$, we have $12 < |L| \leq 24$; where the last inequality follows since if $|L|>24$, then $\mu(\T \circ \T) <12$ and a product of powers of 2 and 3, yet $11 \leq \mu(\T \circ \T)$. This leaves three cases on the order of $L$: either $L$ has order $16, 18$ or $24$.  \vspace{6pt}

First suppose that $L$ has order $16$. Then $L$ is a $2$-group and is thus contained in the unique Sylow $2$-subgroup of $\mathcal{T} \circ \mathcal{T}$, namely $Q_8 \circ Q_8$. Now
$Q_8 \circ Q_8$ is a nilpotent group and $L$ is a subgroup of index $2$, thus is normal in $Q_8 \circ Q_8$. Therefore $L$ contains the centre
of $Q_8 \circ Q_8$ which is also the centre of $\mathcal{T} \circ \mathcal{T}$, however this contradicts the assumption that $\core(L)$ is trivial. \vspace{6pt}

Now suppose that $L$ has order $18$. Then $L$ by Sylow theory, $L$ has a unique Sylow $3$-subgroup, $\Syl_3$ of order $9$, and is thus a semidirect product of $\Syl_3$ by cyclic group subgroup of order $2$. Write $L=\Syl_3 \rtimes \langle w \rangle$. By Lemma \ref{lemma:structure}, any element of order $3$ in $\mathcal{T}  \circ \mathcal{T}$
normalizes $Q_8 \circ Q_8$. On the other hand, $w$ normalizes $\Syl_3$ in $L$. Thus $$[w,\Syl_3]=w^{-1}\Syl_3^{-1}w\Syl_3 \subset (Q_8 \circ
Q_8) \cap \Syl_3=\{1\}.$$ So $w$ in fact must commute with $\Syl_3$ and since the only element which does this is the central involution, we
have $\langle w \rangle= Z(\mathcal{T}  \circ  \mathcal{T})$. Again this contradicts that $\core(L)=\{1\}$. \vspace{6pt}

Finally suppose that $L$ has order $24$. Then $L$ has a Sylow $2$-subgroup of order $8$ which must be contained in a copy of $Q_8 \circ Q_8$ and which is necessarily  core-free subgroup in  $Q_8 \circ Q_8$.  This subgroup would then afford a transitive faithful permutation representation of $Q_8 \circ Q_8$ of degree $4$, contradicting that $\mu(Q_8 \circ Q_8)=8$ by Example \ref{example:quatcentral}.
\vspace{6pt}

Therefore any core-free subgroup of $\mathcal{T}  \circ  \mathcal{T}$ has index at least $24$ which proves $\mu(\T \circ \T)=24$, completing the proof.
\end{proof}

\begin{proposition} \label{proposition:icosa}
We have $\mu(\I \circ \I)=120$.
\end{proposition}

\begin{proof}
The map 
\begin{eqnarray*}
\I \times \I  &\longrightarrow& \Sym(\I); \\ 
(g,h) &\mapsto& \sigma_{(g,h)}: x \mapsto g^{-1}xh, \quad \text{ for all} \, \, x \in \I,
\end{eqnarray*}
yields an embedding of $\I \circ \I$ in $\Sym(\I)$ and so $\mu(\I \circ \I) \leq 120$. On the other hand, since we know that $\mu(\I \circ \I)$ is afforded by a core-free subgroup of minimal index, we seek to prove that  this minimal index is at least $120$. Let $\overline{(-1,1)}$ denote the unique central element of order $2$ in $\I \circ \I$. By Lemma \ref{lemma:icosahedral}, we have a short exact sequence: 
$$\{1\} \longrightarrow \langle \overline{(-1,1)} \rangle \longrightarrow \I \circ \I  \longrightarrow \Alt(5) \times \Alt(5)  \longrightarrow \{1\}.$$ Suppose $L$ is a core-free subgroup of $\I \circ \I$. Then it is isomorphic to its image in $\Alt(5) \times \Alt(5)$ and so we may rephrase the condition on the core-free subgroup affording the minimal degree of $\I \circ \I$ as being the largest subgroup of $\Alt(5) \times \Alt(5)$ such that the above exact sequence splits. We observe that such a subgroup is necessarily core-free in $\Alt(5) \times \Alt(5)$. \vspace{6pt}

Let $A=\{(a,a) \, | \, a \in \Alt(5)\}$, the diagonal subgroup of the direct product. By the simplicity of $\Alt(5)$, it is a core-free subgroup of order $60$. We make the following claim: \\
\\
\underline{Claim:} If $M$ is a subgroup of $\Alt(5) \times \Alt(5)$ of order strictly greater than $60$, then $M$ is not core-free and/or the exact sequence at $M$ does not split.\\
\\ 
Suppose that the order of $M$ is strictly greater than $60$. Observe that $|M|$ divides $60^2=2^4\cdot 3^2 \cdot 5^2.$ \vspace{6pt}

Suppose that $3^2 \cdot 5$ divides the order of $M$. Then $M$ contains a Sylow $3$-subgroup isomorphic to $C_3 \times C_3$ and a Sylow $5$-subgroup of order $5$. Without loss of generality, we may suppose that $(\alpha, 1)$ or $(\alpha, \beta)$ is the generator of the Sylow $5$-subgroup of $M$. Let $(\gamma, 1)$ be an element of order $3$ in the Sylow $3$-subgroup. Then by standard theory of the alternating group, it follows that the subgroup  $\langle (\alpha, \beta), (\gamma,1)\rangle $ of $M$,  contains $\Alt(5) \times \{1\}$ and so $M$ cannot be core-free. \vspace{6pt}

We observe here that if $3\cdot 5^2$ divides the order of $M$, then a similar argument given immediately above demonstrates again that $M$ contains a $\Alt(5) \times \{1\}$ as a subgroup and so cannot be core-free. \vspace{6pt}

Now suppose that $2^3$ divides the order of $M$. Then $M$ contains a Sylow $2$-subgroup isomorphic to $C_2 \times C_2 \times C _2$ since the Sylow $2$-subgroup of $\Alt(5) \times \Alt(5)$ is elementary abelian of rank $4$. Identifying $M$ with its pre-image in $\I \circ \I$ we find that this Sylow $2$-subgroup of $M$ is contained in a copy of $Q_8 \circ Q_8$, a Sylow $2$-subgroup of $\I \circ \I$. Extending $M$ by the centre of $\I \circ \I$ forces an elementary abelian $2$-group of rank $4$ to exist inside $Q_8 \circ Q_8$. But since $Q_8 \circ Q_8$  has order $32$ and contains multiple elements of order $4$, we have a contradiction. A similar but more immediate contradiction is reached when we suppose that $2^4$ divides the order of $M$. \vspace{6pt}

Thus the only remaining case to consider is when $M$ has order $100$. In this case, $M$ contains a copy of $C_5 \times C_5$ as its Sylow $5$-subgroup and since $\Alt(5)$ does not contain any subgroups of order $20$, that is, no $\{2,5\}$-Hall subgroups, it follows that $M$ is isomorphic to $D_5 \times D_5$ (the direct product of two copies of the dihedral group of order $10$).  Now, identifying $M$ with its pre-image in $\I \circ \I$ we find that the elements of order $5$ are normalised by elements of order $2$. However, by the uniqueness of the involution in the quaternion group and elementary Sylow theory, the normaliser of the group of order $5$ in $\I$ is a cyclic group of order $4$ and so the only element of order $2$ in $\I$ that normalises an element of order $5$ is the central element. Thus it follows that the central element of $\I \circ \I$ is contained in $M$ and so $M$ cannot be core-free. \vspace{6pt}

This completes the proof of the claim and so $A$ is the largest subgroup of $\Alt(5) \times \Alt(5)$ such that the exact sequence at $A$ splits. Therefore, identifying $A$ with is pre-image once more, we have proved   that $\mu(\I \circ \I)=|\I \circ \I:A|=120$. 
\end{proof}

\begin{remark} 
In calculating the minimal degrees of $\T \circ \T$ and $\I \circ \I$, we have exhibited examples of groups that have quotients whose minimal degree is larger than that of the group itself. Thus the groups $\T \times \T$ and $\I \times \I$ are {\it exceptional} (not to be confused with the Coxeter group sense of the word) and the central products are called {\it distinguished}. We will show in Section \ref{section:exceptional} that these groups fall into a more general framework and produce similar examples there. 
\end{remark}

\subsection{The Minimal Degrees of $W(F_4)$ and $W(H_4)$}

We recall a well-known fact that every Coxeter group acts faithfully on its associated root system. Moreover when the Coxeter group is
irreducible, roots of any given length are contained in a single orbit, on which the group also acts faithfully (see \cite{H90}). Thus the size
of any orbit in the root system is always an upper bound for the minimal degree of a Coxeter group. It is the case that for $W(H_4)$ and
$W(F_4)$, we cannot do any better than the size of their root systems for their minimal degrees. Below we state convenient structures of
$W(H_4)$ and $W(F_4)$ that allow us to calculate their minimal degrees. A proof of these structures can be found in \cite{TL07}.

\begin{proposition} \label{proposition:abstractisomorphism} 
We have these abstract isomorphisms;
\begin{align*}
W(H_4) & \cong (\mathcal{I} \circ \mathcal{I}) \rtimes C_2,  \\
 W(F_4) & \cong (Q_8 \circ Q_8) \rtimes (\Sym(3) \times \Sym(3)) \cong (\mathcal{T} \circ \mathcal{T}) \rtimes (C_2 \times C_2),
 \end{align*} under appropriate actions.
\end{proposition}

Now the root system of the Coxeter group $W(F_4)$ consists of $48$ roots of two lengths: $24$ long and $24$ short roots (see \cite{H90}), thus $\mu(W(F_4)) \leq 24$. On the other hand by Proposition \ref{proposition:tetra} and Proposition \ref{proposition:abstractisomorphism}, we have  $24=\mu(\mathcal{T} \circ\mathcal{T}) \leq \mu(W(F_4))$ and so we have: 

\begin{proposition}
The minimal degree of $W(F_4)$ is $24$.
\end{proposition}

Similarly, the root system of the Coxeter group $W(H_4)$ consists of $120$ roots of equal length and are all contained in the same orbit, so $\mu(W(H_4)) \leq 120$. On the other hand, by Proposition  \ref{proposition:abstractisomorphism} and Proposition \ref{proposition:icosa}, we have $120 =\mu(\I \circ \I) \leq \mu(W(H_4))$  and so we have: 

\begin{proposition}
The minimal degree of $W(H_4)$ is $120$.
\end{proposition}

\section{The Groups $W(E_6), W(E_7)$ and $W(E_8)$} \label{section:exceptionalcoxeter} 

In this section we will use the fact that for every Coxeter group, there is a well-defined length function which induces a {\it sign}
homomorphism $W \longrightarrow \{\pm 1\}; w \mapsto (-1)^{l(w)}$. The kernel of this homomorphism is an index $2$ subgroup denoted by $W^{+}$ called
the {\it rotation} subgroup. It is the case that for the groups $W(E_6)$ and $W(E_7)$, their rotation subgroups are simple. Now calculating
the minimal permutation degree for a simple group reduces to finding the maximal subgroups, since $\mu(S)$ for $S$ a simple group is furnished by a maximal subgroup of smallest index. This area has been well-studied and so we do not reproduce the proofs here, but will refer to \cite{KL90} when needed. \vspace{6pt}

Let us first deal with the group $W(E_6)$. By Humphreys \cite[Section 2.12]{H90}, its rotation subgroup $W(E_6)^+$ is a simple group isomorphic to
$SU_{4}(\mathbb{F}_2)$ and by \cite[Table 5.2.A]{KL90}, it has minimal degree $27$. On the other hand, $W(E_6)$ acts faithfully on the set of
positive/negative roots of $E_7$ that are not contained in $E_6$. By inspection of the size of the root systems, this set has size $27$ as well.
Therefore the minimal degrees of $W(E_6)$ and its rotation subgroup co-inside, both being $27$. \vspace{6pt}

We now turn to $W(E_7)$. Again by  \cite[Section 2.12]{H90}, this group is a split extension of its rotation subgroup by a cyclic group of order
$2$. Now its rotation subgroup $W(E_7)^+$ is a simple group isomorphic to $O_{7}(\mathbb{F}_2) \cong Sp_{6}(2)$, and so again by \cite[Table 5.2.A]{KL90}, its
minimal degree is $28$. Thus we have the following decomposition $$W(E_7) \cong O_{7}(\mathbb{F}_2) \times C_2,$$ where each direct factor is a
simple group. So by Theorem \ref{theorem:simplegroups} we have $\mu(W(E_7))=28+2=30$.\vspace{6pt}

Before we deal with the group $W(E_8)$, we require a lemma about covering groups.

\begin{lemma} \label{lemma:centralextension}
For $p$ a prime, let $G$ be a $p:1$ non-split central extension of the simple group $S$. Then $$\mu(G) \geq p\mu(S).$$
\end{lemma}

\begin{proof}
Let $\varphi$ map $G$ surjectively onto $S$ with kernel $C_p$. It follows that this kernel is the unique
minimal normal subgroup of $G$, for if $N$ is a non-trivial proper normal subgroup of $G$ distinct from $\ker(\varphi)$, then $\varphi(N)$ is a
non-trivial proper normal subgroup of $S$ (since $\varphi$ is non-split), contradicting simplicity. Therefore $\mu(G)$ equals index of the largest subgroup which
does not contain the kernel. Let this subgroup be $L$ and observe that $L$ is isomorphic to its image in $S$ under $\varphi$. Therefore we have
$$\mu(G)=|G:L|=p|S:\varphi(L) |\geq p\mu(S),$$ where the last inequality is necessary since $ \varphi(L)$ need not be maximal in $S$.
\end{proof}

Now by \cite[Section 2.12]{H90} once again, the rotation subgroup of $W(E_8)$ is a $2:1$ non-split central extension of the simple group $O_{8}^{+}(\mathbb{F}_2)$. By \cite[Table
5.2.A]{KL90} this group has minimal degree $120$ and so by Lemma \ref{lemma:centralextension}, $$\mu(W(E_8)^{+}) \geq
2\mu(O_{8}^{+}(\mathbb{F}_2))=240.$$ On the other hand, the root system of $W(E_8)$ has size $240$ as well and so we deduce that
$\mu(W(E_8)^{+})=\mu(W(E_8))=240$.\vspace{6pt}

Summarising this for the exceptional Coxeter groups of type $E$: 

\begin{proposition} 
We have \begin{enumerate}
\item $\mu(W(E_6))=27$,
\item $\mu(W(E_7))=30$,
\item $\mu(W(E_8))=240$.
\end{enumerate}
\end{proposition}

\section{The Groups $W(I_{2}(m))$}

For $m\geq 5$, the groups $W(I_{2}(m))$ are isomorphic to the dihedral groups of order $2m$. The minimal degrees of dihedral groups were calculated by Easdown and
Praeger in \cite{EP88}: we include the full statement of their theorem even though for some values of $n$ are $r$ below, we do not obtain an irreducible Coxeter group. 

\begin{proposition}\cite[Proposition 2.8]{EP88}
For any integer $k=\prod_{i=1}^{m}p_{i}^{\alpha_i} >1$, with the $p_i$ distinct primes, define $\psi(k)=\sum_{i=1}^{m}p_{i}^{\alpha_i}$, with
$\psi(1)=0$. Then for the dihedral group $D_{2^{r}n}$ of order $2^{r}n$, with $n$ odd, we have
\begin{displaymath}
\mu(D_{2^{r}n})= \left\{ \begin{array}{lll}
2^r &\text{if} \quad n=1, 1\leq r \leq 2 \\
2^{r-1} & \text{if} \quad  n=1, r>2 \\
\psi(n) & \text{if} \quad  n>1, r=1 \\
2^{r-1} + \psi(n) & \text{if} \quad  n \geq 1, r>1.\\
\end{array} \right.
\end{displaymath}
\end{proposition}

\section{Direct Products} \label{section:directproduct}

Recall that an arbitrary finite Coxeter group is the direct product of irreducible Coxeter groups. While in many areas of study it is enough just to
concentrate on the irreducible components, this is not the case when we wish to consider minimal degrees of arbitrary finite Coxeter groups.  \vspace{6pt}

We can make the following observations with what is known about minimal degrees of direct products. The groups $W(A_n)$ and $W(B_n)$ are
contained in the Wright class $\mathscr{C}$ (recall from Section \ref{section:background}, these groups contain a nilpotent subgroup of the same minimal degree) and so
$$\mu(W(A_n) \times W(B_n))=\mu(W(A_n)) + \mu(W(B_n)).$$ Also since $W(H_3)$ and $W(E_7)$ are direct products of simple groups, we have by
Theorem \ref{theorem:simplegroups} $$\mu(W(H_3) \times W(E_7))=\mu(W(H_3)) + \mu(W(E_7)).$$

On the other hand, we have the following isomorphism for odd $n$ greater than or equal to $5$; $W(D_n) \times W(A_1) \cong W(B_n)$. So we have
$$\mu(W(D_n) \times W(A_1))=\mu(W(B_n))$$ but $\mu(W(D_n))+\mu(W(A_1))=2n+2 > 2n=\mu(W(B_n))$.  \vspace{6pt}

Thus the minimal degree of an arbitrary Coxeter group is not in general equal to the sum of the minimal degrees of its irreducible components.\vspace{6pt}

\begin{remark} 
The motivation for considering the minimal degree of irreducible Coxeter groups was to produce further examples of \eqref{eq:directsum} being a strict inequality, that is,  groups $G$ and $H$ that satisfy the inequality 
$$\mu(G \times H) < \mu(G) + \mu(H).$$ All of our examples thus far have had the property that $\mu(G)=\mu(G \times H)$, which is always a sufficient condition. 
We are currently unaware of any examples of groups $G$ and $H$ that are not themselves non-trivial direct products and satisfy the cascade of strict inequalities  \begin{equation}  \label{eqn:cascade} \max\{\mu(G), \mu(H)\} < \mu(G \times H) < \mu(G) + \mu(H). \end{equation} Here we stipulate that $G$ and $H$ are not non-trivial direct products for the following reason: let $G=W(D_5)$ and $H=C_2 \times C_2$, then $\max\{\mu(G), \mu(H)\}=10$ and $\mu(G)+\mu(H)=14$. However,
\begin{eqnarray*}
G \times H &\cong& ((C_2 \times C_2 \times C_2 \times C_2) \rtimes \Sym(5)) \times (C_2 \times C_2)  \\
		&\cong & ((C_2 \times C_2 \times C_2 \times C_2 \times C_2) \rtimes \Sym(5)) \times C_2 \\
		& \cong & W(B_5) \times C_2.
\end{eqnarray*}
Therefore $\mu(G \times H)=12$ and so we do get an example of \eqref{eqn:cascade}, however it was manufactured from an example of \eqref{eq:directsum} being a strict inequality. In fact, one can easily manipulate any example where \eqref{eq:directsum} is a strict inequality, to provide an example that satisfies \eqref{eqn:cascade}. Thus we seek an example of groups $G$ and $H$ satisfying \eqref{eq:directsum} and the extra condition that $\max\{\mu(G),\mu(H)\} < \mu(G \times H).$
\end{remark}

\section{New Examples of Distinguished Quotients} \label{section:exceptional}
In this section, we exhibit further examples of exceptional groups motivated by our calculations of the minimal degrees of the binary tetrahedral and icosahedral groups in Section \ref{section:realreflections}. \vspace{6pt}

We return to the quaternions to define the binary octahedral group and calculate its minimal degree. Let $\gamma:=\frac{1}{\sqrt{2}}(1+i)$. Then $\gamma$ has order $8$ and from our description of $\mathcal{T}$ in Subsection \ref{subsection:polyhedral}, it is easily verified that
$\gamma$ normalises both $Q_8$ and $\mathcal{T}$. The group $\mathcal{O}:=\mathcal{T}\langle \gamma \rangle$ generated by $\mathcal{T}$ and
$\gamma$ is called the {\it binary octahedral group}. Since $\gamma^2=i$, it follows that $\mathcal{O}$ can be generated by $\omega$ and
$\gamma$ and has order $48$. A presentation for $\mathcal{O}$ can be given thus, $$\langle o_1, o_2, o_3 \ | \
o_{1}^{4}=o_{2}^{3}=o_{3}^{2}=o_{1}o_{2}o_{3} \rangle$$ or equivalently, since $o_{3}^{2}=o_{1}o_{2}o_{3}$ implies $o_{3}=o_{1}o_{2}$, we have
$$\langle o_1, o_2 \ | \ o_{1}^4=o_{2}^3=(o_{1}o_{2})^2 \rangle,$$ by identifying $o_{1}$ with $\gamma$ and $o_2$ with $-\omega^2$ (see \cite[page 17]{Su07}).

\begin{lemma}
We have $\mathcal{O}/\{\pm 1\} \cong \Sym(4)$. 
\end{lemma}

\begin{proof}
We may give $\mathcal{O}/\{\pm 1\}$ the presentation $$\langle o_1, o_2 \ | \ o_{1}^4=o_{2}^3=(o_{1}o_{2})^2 =1\rangle,$$ which is the well-known presentation for $\Sym(4)$.
\end{proof}

\begin{proposition}
We have $\mu(\mathcal{O})=16$, with an explicit representation on the right cosets of $\langle \omega \rangle$:
\begin{eqnarray*}
\mathcal{O} &\hookrightarrow& \Sym(16); \\
    o_1       &\mapsto&  (1 \, 3 \, 7 \, 6 \, 2 \, 5 \, 10 \, 4)(8 \, 13 \, 15 \, 11 \, 9 \, 12 \, 16 \, 14), \\
    o_2       &\mapsto& (1 \, 2)(3 \, 4 \, 9 \, 5 \, 6 \, 8)(7 \, 11 \, 13 \, 10 \ 14 \, 12)(15 \, 16).
\end{eqnarray*}
\end{proposition} 

\begin{proof}
We observe that $\gamma^4=-1$ and that  $j^{-1}\gamma j = \gamma^{-1}$, and so $\langle
\gamma, j \rangle \cong Q_{16}$. By Example \ref{example:quaternions}, $\mu(Q_{16})=16$ and so $16 \leq \mu(\mathcal{O})$. On the other
hand,  noting that $\mathcal{O}$ can be generated by $\omega$ and $\gamma$, recalling $\omega= \tfrac{1}{2}(-1+i+j+k)$ is an element of order $3$, we explicitly list the elements of the following groups:
 \begin{eqnarray*}
  \langle \omega \rangle &=& \big\{ 1, \tfrac{1}{2}(-1+i+j+k), \tfrac{1}{2}(-1-i-j-k) \big\}, \\
   \langle \gamma \omega  \gamma^{-1} \rangle &=& \big\{ 1, \tfrac{1}{2}(-1+i+j-k), \frac{1}{2}(-1-i-j+k) \big\}.
 \end{eqnarray*} Thus it immediately follows that  
$\langle \omega \rangle \cap \langle \omega^{\gamma} \rangle= \{1\}$ and so $\langle \omega \rangle$ is a core-free subgroup of
index $16$. Therefore $\mu(\mathcal{O})=16$. 
\end{proof}

Since $Q_{16}$ is a proper subgroup of $\mathcal{O}$ with $\mu(Q_{16})=\mu(\mathcal{O})=16$, $\mathcal{O}$ is contained in Wright's class, thus we immediately have:

\begin{corollary} 
We have $\mu(\mathcal{O} \times \mathcal{O})= 32$.
\end{corollary}

\begin{remark} \label{remark:normalise2}
It can be shown using elementary Sylow theory that the normaliser of $\langle \omega \rangle$ in $\mathcal{O}$ has order $12$ and is therefore isomorphic to $C_3 \rtimes C_4$ since $Q_8$ has a unique involution. This simple fact will be used when calculating the minimal degree of $\mathcal{O} \circ \mathcal{O}$. Again since we know in advance that we seek a core-free subgroup of least index, the following can easily be verified on a computer. 
\end{remark}

\begin{proposition}
The group $\mathcal{O} \times \mathcal{O}$ is exceptional with distinguished quotient $\mathcal{O} \circ \mathcal{O}$ which has minimal degree $48$.
\end{proposition}
\begin{proof}
The now familiar map $\mathcal{O} \times \mathcal{O} \longrightarrow \Sym(\mathcal{O})$ sending the pair $(g,h)$ to the permutation that maps $x \mapsto g^{-1}xh$ for all $x$ in $\mathcal{O}$ once again yields an embedding of $\mathcal{O} \circ \mathcal{O}$ in $\Sym(\mathcal{O})$, so $\mu(\mathcal{O} \circ \mathcal{O}) \leq 48$. Observe that $Q_{16} \circ Q_{16}$ is a subgroup whose minimal degree is $16$, so $16 \leq \mu(\mathcal{O} \circ \mathcal{O}) \leq 48$. Since $\mathcal{O} \circ \mathcal{O}$ has a unique minimal normal subgroup generated by the central element of order $2$, it suffices to show that any core-free subgroup must have index at least $48$. If $L$ is such a subgroup, then since $\mathcal{O} \circ \mathcal{O}$ has order $2^7 \cdot 3^2$, a quick calculation shows that $24 \leq |L| \leq 72$. As in the proof of Proposition \ref{proposition:icosa}, let  $\overline{(-1,1)}$ denote the central element of order $2$ and consider the short exact sequence:
$$\{1\} \longrightarrow \langle \overline{(-1,1)} \rangle \longrightarrow \mathcal{O} \circ \mathcal{O}  \longrightarrow \Sym(4) \times \Sym(4)  \longrightarrow \{1\}.$$ We know that $\mu(\mathcal{O} \circ \mathcal{O}) $ is given by the largest subgroup $L$ of $\Sym(4) \times \Sym(4)$ such that the above exact sequence splits at $L$; such a subgroup is necessarily core-free in $\Sym(4) \times \Sym(4)$. For what follows it will be convenient think of $\Sym(4)$ via its well-known permutation representation that we gave in the proof of Theorem \ref{theorem:Dn}. Thus the direct product $\Sym(4) \times \Sym(4)$ is abstractly isomorphic to \begin{equation} \label{equation:sym4}  (N_1 \times N_2) \rtimes (\Sym(3) \times \Sym(3)), \end{equation}  under the appropriate actions, where $N_1 \cong N_2 \cong C_2 \times C_2$. We make the following claim:\\
\\
\noindent \underline{Claim:} If $L$ is a subgroup of $\Sym(4) \times \Sym(4)$  of order strictly greater than $24$, then $L$ is not core-free and/or the above exact sequence does not split.\\
\\ 
Suppose for a contradiction that $L$ has order strictly greater than $24$. If $2^4$ divides the order of $L$, then $L$ has a Sylow $2$-subgroup of order $2^4$. Identifying $L$ with its pre-image in $\mathcal{O} \circ \mathcal{O}$ forces this Sylow $2$-subgroup to be contained in $Q_{16} \circ Q_{16}$ and moreover to be core-free in it as well. Therefore this Sylow $2$-subgroup would afford a faithful representation of  $Q_{16} \circ Q_{16}$ degree $8$ which contradicts $\mu(Q_{16} \circ Q_{16})=16$. A similar contradiction is reached when we assume $2^5$ divides the order of $L$. Thus the remaining cases to consider are $|L|=72$ and $|L|=36$. \vspace{6pt}

Suppose that $|L|=72$. Then $L$ contains a Sylow $2$-subgroup of order $8$ and a Sylow $3$-subgroup of order $9$, which must be isomorphic to $C_3 \times C_3$ and contained in the top group of \eqref{equation:sym4}. Now, by a simple observation on the order we find that the Sylow $2$-subgroup of $L$ intersects non-trivially with the base group $N_1 \times N_2$. Let $n_1n_2$ be an element in this intersection, where $n_1 \in N_1$ and $n_2 \in N_2$, and let $\alpha$ be an element of order $3$ which commutes with $N_2$, which necessarily exists. Now, it can be readily seen that $n_{1}^{\alpha} \neq n_1$ and so $\langle n_1n_2, (n_1n_2)^\alpha \rangle$ is isomorphic to  $ C_2 \times C_2 \times C_2$ and so must contain the base group $N_1$. Thus $L$ cannot be core-free. \vspace{6pt}

Now suppose that $|L|=36$. In this case, the Sylow $2$-subgroup need not intersect the base group, but since $L$ must contain a copy of $C_3 \times C_3$, $L$ is forced to be the top group $\Sym(3) \times \Sym(3)$ of \eqref{equation:sym4}. Identifying $L$ with its pre-image in $\mathcal{O} \circ \mathcal{O}$ we find that the elements of order $3$ are normalised by elements of order $2$. However by Remark \ref{remark:normalise2}, the only element to do this in $ \mathcal{O}$ is the central element. Thus $\overline{(-1,1)}$ is contained $L$ and so $L$ is not core-free. \vspace{6pt}

This verifies the claim and so the largest core-free subgroup of $\mathcal{O} \circ \mathcal{O}$ has index $48$ proving $\mu( \mathcal{O} \circ \mathcal{O})=48$, and so $\mathcal{O} \times \mathcal{O}$ is an exceptional group with is thus a distinguished quotient the central product.
 \end{proof}

The following theorem is due to Easdown and Praeger.

\begin{theorem}\cite[Theorem 2.1]{EP88} \label{theorem:EP}
Let $p$ be a prime and let $G_1$ and $G_2$ be non-cyclic $p$-groups with centres $Z_1= \langle a_1 \rangle$ and $Z_2= \langle a_2 \rangle$
respectively of the same order. Then,
\begin{enumerate}
\item[(i)] if $p$ is odd, then $G_1 \times G_2$ is an exceptional group with distinguished subgroup $N=\langle (a_1,a_2) \rangle$; 
\item[(ii)] if $p=2$, $\mu(G_1)>\mu(G_2)$ and $G_2$ is not a generalised quaternion group, then $G_1 \times G_2$ is an exceptional group with
distinguished subgroup $N=\langle (a_1,a_2) \rangle$.
\end{enumerate}
\end{theorem}

Following our calculations in the above proposition and in Propositions \ref{proposition:tetra} and \ref{proposition:icosa},  we now extend the above theorem to include nilpotent groups of odd order with cyclic centres. More generally, we can apply this central product construction to groups $G$ in
Wright's class $\mathscr{C}$ whose nilpotent subgroup $G_1$ has a cyclic centre that coincides with the centre of $G$.

\begin{proposition} \label{proposition:quotient}
Let $G$ and $H$ be groups in $\mathscr{C}$. Suppose that $G_1$ and $H_1$ are the corresponding nilpotent groups of $G$ and $H$ respectively,
such that they have odd order with
$$Z(G)=Z(G_1)\cong Z(H)=Z(H_1) \cong \prod_{i=1}^{r}C_{p_{i}^{\alpha_{i}}},$$ for distinct odd primes $p_i$. Express $G_1=\prod_{i=1}^{r}G_{p_{i}}$ and $H_1=\prod_{i=1}^{r}H_{p_{i}}$, where the $G_{p_{i}}$
and the $H_{p_i}$ are the Sylow $p_i$-subgroups of $G_1$ and $H_1$ respectively. Then $G \times H$ is exceptional with distinguished quotient
the central product $G \circ H$.
\end{proposition}

\begin{proof}
We first observe that $G_1 \circ H_1 \leq G \circ H$ since $Z(G)=Z(G_1)$ and $Z(H)=Z(H_1)$. Moreover, since $\mu(G_1 \times H_1)=\mu(G \times
H)$, it suffices to show that $G_1 \times H_1$ is exceptional with distinguished quotient $G_1 \circ H_1$. \vspace{6pt}

For each $g \in G_1$ we may write $g=g_{1}\ldots g_{r}$ where $g_i \in G_{p_i}$ and similarly for $h \in H_1$ we have $h=h_{1}\ldots h_{r}$
where $h_i \in H_{p_i}$. This yields an isomorphism
\begin{eqnarray*}
G_1 \times H_1 &\longrightarrow& \prod_{i=1}^{r}(G_{p_i} \times H_{p_i}) \\
 (g,h) &\mapsto& \big((g_1, h_1),\ldots,(g_r, h_r)\big).
 \end{eqnarray*}
Since $Z(G_1) \cong Z(H_1) \cong \prod_{i=1}^{r}C_{p_{i}^{\alpha_{i}}}$, for simplicity of notation we will let $Z=\prod_{i=1}^{r}Z_{p_{i}}$
simultaneously denote $Z(G_1)$ and $Z(H_1)$, and write for $z \in Z$, $z=z_1\ldots z_r$ with $z_i \in Z_{p_{i}}$. Thus for each $i$, we may
write $Z(G_{p_i})=Z(H_{p_i})=Z_{p_{i}} \cong C_{p_{i}^{\alpha_i}}$, and so we may form the central product $$G_{p_{i}} \circ
H_{p_{i}}=(G_{p_{i}} \times H_{p_{i}})/\overline{Z_{p_{i}} \times Z_{p_{i}}},$$ where $\overline{Z_{p_{i}} \times Z_{p_{i}}}=\langle(z_i, z_i) \
| \ z_i \in Z_{p_{i}} \rangle \cong C_{p_{i}^{\alpha_{i}}}$. This gives rise to a map

\begin{eqnarray*}
\eta: G_1 \times H_1 &\longrightarrow& \prod_{i=1}^{r}(G_{p_i} \circ H_{p_i}) \\
 (g,h) &\mapsto& \big(\overline{(g_1, h_1)},\ldots,\overline{(g_r, h_r)}\big).
 \end{eqnarray*}

\noindent \underline{\textbf{Claim:}} \quad $\ker \eta = \overline{Z \times Z}=\langle (z,z) \ | \ z \in Z \rangle$. \\

Let $z=z_1 \ldots z_r \in Z$, then $$\eta(z,z) = \big( \overline{(z_1, z_1)}, \ldots, \overline{(z_r, z_r)} \big)= \big( \overline{(1, 1)},
\ldots, \overline{(1,1)} \big),$$ so $\overline{Z \times Z} \subseteq \ker \eta$. Now suppose that $(g,h) \in \ker \eta$. So $\eta(g,h)  = \big(
\overline{(1, 1)}, \ldots, \overline{(1,1)} \big),$ and so writing $g=g_1\ldots g_r$ and $h=h_1\ldots h_r$, we find that $$\overline{(g_1,
h_1)}=\overline{(1, 1)}, \ldots, \overline{(g_r, h_r)}=\overline{(1,1)}.$$ Therefore, $$(g_1, h_1)=(z_1, z_1), \ldots, (g_r, h_r)=(z_r, z_r),$$
and we may write $(g, h)=(z, z)$ where $z=z_{1}\ldots z_{r}$ with $z_i \in Z_{p_{i}}$. Therefore $(g, h) \in \overline{Z \times Z}$ proving
$\ker \eta \subseteq \overline{Z \times Z}$, whence the claim.\vspace{6pt}

We therefore have an isomorphism $G_1 \circ H_1 \cong \prod_{i=1}^{r}(G_{p_i} \circ H_{p_i})$. By Theorem \ref{theorem:EP}, $\mu(G_{p_{i}} \circ
H_{p_{i}}) > \mu(G_{p_{i}} \times H_{p_{i}})$ for each $i$, and since each $G_{p_{i}} \circ H_{p_{i}}$ lies in Wright's class $\mathscr{C}$,  we have
$$\mu(G_1 \circ H_1)= \sum_{i=1}^{r} \mu(G_{p_{i}} \circ H_{p_{i}}) > \sum_{i=1}^{r} \mu(G_{p_{i}} \times H_{p_{i}})=\mu(G_1 \times H_1).$$ So
$G_1 \times H_1$ is exceptional with distinguished quotient $G_1 \circ H_1$. Since $G_1 \circ H_1 \leq G \circ H$, we have $$\mu(G \circ H) \geq
\mu(G_1 \circ H_1) > \mu(G_1 \times H_1)=\mu(G \times H),$$ which completes the proof of the proposition.
\end{proof}

We now give the following concrete examples which once again involve the monomial reflection groups $G(d,e,n)$. We very briefly describe the relevant structure properties of these groups (see \cite{TL07} for a compressive treatment of these groups). \vspace{6pt}

Let $d,e$ and $n$ be positive integers with $e$ dividing $d$. Let $C_{d}$ be a cyclic group of order $d$ and let $A(d,e,n)$ be the subgroup of the direct product of $n$ copies of $C_d$ defined as follows: $$A(d,e,n):= \{ (\theta_1, \theta_2,  \ldots,  \theta_{d}) \, | \, (\theta_{1}\theta_{2}\ldots\theta_{n})^{\tfrac{d}{e}}=1 \}.$$  This group naturally comes equipped with an action of the symmetric group $\Sym(n)$ by permuting the coordinates and thus the group $G(d,e,n)$ is then defined as the following semidirect product $$G(d,e,n):=A(d,e,n) \rtimes \Sym(n).$$ This is a normal subgroup of the full wreath product $C_d \wr \Sym(n)$. 

\begin{example}

Let $p$ be an odd prime. From our description above, the group $G(p,1,p)$ is the full wreath product of the cyclic group $C_p$ with the symmetric group $\Sym(p)$. From
now on, denote the base group of the wreath product $C_p \wr \Sym(p)$ by $A$ and let $\theta_{1}, \ldots , \theta_{p}$ be the standard
generators. \vspace{6pt}

Let $W$ denote $G(p,1,p)$ and let $b:=(1 \ 2 \ \ldots \ p) \in \Sym(p)$. Define a proper subgroup of $W$, $$H:=A\langle b \rangle= A \rtimes
\langle b \rangle \cong (\underbrace{C_p \times \ldots \times C_p}_{p}) \rtimes C_p.$$ Clearly $H$ is a $p$-group and a little calculation shows
that the centres of $W$ and $H$ coincide at $\langle \theta_{1}\ldots\theta_{p} \rangle \cong C_p$. Now it is easy to see that $\mu(W) \leq p^2$, and by our first example, Theorem \ref{theorem:abelian}, we have  $$p^2=\mu(A)\leq \mu(H) \leq \mu(W) \leq p^2.$$ Thus $\mu(W)=\mu(H)=p^2$. Therefore this class of groups satisfies the conditions of
Proposition \ref{proposition:quotient} with $H$ as the nilpotent subgroup and so we have $\mu(H \circ H)
> \mu(H \times H)$. Therefore $H \times H$ and $W \times W$ are exceptional groups with distinguished quotients $H \circ H$ and $W \circ W$ respectively. 
\end{example}

\begin{example}
Again let $p$ be an odd prime and let $W$ be the full wreath product $C_p \wr \Sym(p) \cong G(p,1,p)$ as above. Let $G$ be the group $G(p,p,p)$, which is a normal subgroup of index $p$ inside $G(p,1,p)$. Let $b=(1 \ 2 \ \ldots \ p) \in \Sym(p)$ and define a proper subgroup $$K:=A(p,p,p) \rtimes \langle b \rangle \cong (\underbrace{C_p
\times \ldots \times C_p}_{p-1}) \rtimes C_p.$$ From the previous example, $\mu(W) =p^2$ and by \cite[Proposition 3.13]{S08} we again have
$\mu(G)=\mu(K)=p^2$. \vspace{6pt}

We claim that the centres of $W, G$ and $K$ coincide. To see this, let $\theta_1, \theta_2, \ldots, \theta_p$ be the standard generators of the
base group $W$. Then a generating set for $A:=A(p,p,p)$ is given by $$c_1=\theta_{1}\theta_{2}^{-1}, c_2=\theta_{2}\theta_{3}^{-1}, \ldots,
c_{p-1}=\theta_{p-1}\theta_{p}^{-1}.$$ It is clear that $\theta:=\theta_{1}\theta_{2}\ldots\theta_{p}$ is stable under all permutations of the
symmetric group and in fact $\langle \theta \rangle=Z(W)\cong C_p$. Observe also that $\theta=c_1c_{2}^2c_{3}^3\ldots c_{p-1}^{p-1}$ and so
$\theta$ is also contained in both $G$ and $K$. Moreover, it is easily confirmed that the centres of $G$ and $K$ are also cyclic of order $p$,
and thus generated by $\theta$. Therefore, $Z(W)=Z(G)=Z(K)$. Since $K$ is a $p$-group, we have by Proposition \ref{proposition:quotient}, $K \times K$
and $G \times G$ are exceptional groups.
\end{example}

\section{Acknowledgements}

The author sincerely thanks his Ph.D. supervisor David Easdown for his support and guidance; many of the above results were forged in our supervisory meetings. Sincere thanks are given  to Anthony Henderson for many informative discussions on reflections groups and for helpful comments in the earlier stages of this article. We also acknowledge the assistance of the computational algebra system Magma, \cite{BCFS10} in explicitly computing minimal degrees of the binary polyhedral groups. 


\bibliographystyle{plain}

\vspace{6pt}

\end{document}